\documentclass[11pt,a4paper]{article}
\usepackage[pdftex]{graphicx}
\usepackage{amsmath,amssymb,amsfonts,amsthm}
\numberwithin{equation}{section}
\usepackage{indentfirst}
\usepackage{enumitem} 
\usepackage[margin=1.3in]{geometry}
\usepackage{fancyhdr}

\usepackage[colorlinks=true,linkcolor=magenta,citecolor=blue,urlcolor=cyan]{hyperref}
\usepackage{titlesec}
\usepackage{etoolbox}
\usepackage{graphicx}  
\usepackage{changepage}
\theoremstyle{plain}
\newtheorem{theorem}{Theorem}[section]
\newtheorem{lemma}[theorem]{Lemma}
\newtheorem{corollary}[theorem]{Corollary}

\newtheorem{remark}{Remark}[section]

\makeatletter
\renewcommand{\maketitle}{
	\begin{center}
		{\Large\bfseries{\@title}\par}
		\vskip 1em
		{\normalsize
			\lineskip .5em
			\begin{tabular}[t]{c}
				\@author
			\end{tabular}\par}
		\vskip 1.5em
	\end{center}
}
\makeatother
\renewenvironment{abstract}{
	\begin{adjustwidth}{1.3cm}{1.3cm}
		\noindent{\large\bfseries{A{\scriptsize BSTRACT.}}}
	}{
	\end{adjustwidth}
}

\usepackage{color}
\usepackage{xcolor}
\usepackage[normalem]{ulem} 
\usepackage{soul}

\usepackage{graphicx} 
\usepackage{xstring}  

\usepackage{xstring}
\usepackage{graphicx} 

\usepackage{marvosym}

\begin{document}
	
	\title{Arithmetic Properties of Colored Partitions Restricted by Parity of the Parts}
	
	\author{M. P. Thejitha, James A. Sellers, and S. N. Fathima}
	
	\maketitle
	
	\begin{abstract}
		Let $a_{r,s}(n)$ denote the number of mutlicolored  partitions of $n$, wherein both even parts and odd parts may appear in one of $r$-colors and $s$-colors, respectively, for fixed $r,s\ge 1$. The paper aims to study arithmetic properties satisfied by $a_{r,s}(n)$, using elementary generating function manipulations and classical $q$-series techniques.\\
		
		
	\end{abstract}
	
	\bigskip
	
	\vspace{0.5em}
	\section{Introduction}	
	An integer partition $\lambda$ is a weakly decreasing sequence of positive integers. These positive integers are called the parts of $\lambda$. We denote by $|\lambda|$, the sum of all parts of $\lambda$. If $|\lambda|=n$, then $\lambda$ is called a partition of $n$. The generating function for the number of unrestricted partitions of $n$, denoted by $p(n)$, is given by
	\begin{align*}
		\sum_{n=0}^{\infty}p(n)q^n= \dfrac{1}{(q;q)_\infty}.
	\end{align*}
	Here and throughout this paper, we assume $|q|<1$ and for positive integer $n$, we use the standard notation
	\begin{align*}
		(a;q)_0:=1,\;\;(a;q)_n:=\prod_{j=0}^{n-1}(1-aq^j),	\;\; \text{and} \;\;(a;q)_\infty = \prod_{n=0}^\infty(1-aq^{n}).
	\end{align*}
	A partition of a positive integer $n$ has $k$ colors if there are $k$ copies of each part size available and all of them are viewed as distinct objects. A partition of a positive integer into parts with colors is called a ``colored partition". We denote by $p_k(n)$, the number of $k$-colored partitions of $n$, and agree that $p_k(0)=1$. The generating function of $p_k(n)$ satisfies the identity
	\begin{align*}
		\sum_{n=0}^{\infty}p_k(n)q^n=\dfrac{1}{f_1^k}\; ,
	\end{align*}
	where $ f_m:= (q^m;q^m)_\infty$.\\
	\indent	In a series of papers, Hirschhorn and Sellers \cite{hs1}, Sellers \cite{sellers}, and Thejitha and Fathima \cite{thfa} studied congruence properties of partitions wherein even parts come in only one color, while the odd parts may appear in one of $s$ colors for fixed $s\ge1$. The number of such partitions of $n$ is denoted by $a_s(n)$, and it is an easy combinatorial exercise \cite{hs1} to determine that the generating function for $a_s(n)$ is
	\begin{align*}
		\sum_{n=0}^{\infty}a_s(n)q^n=\dfrac{f_2^{s-1}}{f_1^s}\; .
	\end{align*}
	In \cite{hs1}, Hirschhorn and Sellers  proved the following families of congruences modulo 7 using elementary $q$-series identities and straightforward generating function manipulations.
	\begin{theorem}[{{\cite[Corollary~3.1]{hs1}}}]
		For $j\ge0$ and all $n\ge0$,
		\begin{align*}
			a_{7j+1}(7n+5)\equiv 0 \pmod{7}\\
			a_{7j+3}(7n+2)\equiv 0 \pmod{7}\\
			a_{7j+4}(7n+4)\equiv 0 \pmod{7}\\
			a_{7j+5}(7n+6)\equiv 0 \pmod{7}\\
			a_{7j+7}(7n+3)\equiv 0 \pmod{7}.
		\end{align*}
	\end{theorem}
	In \cite{thfa}, the authors employed the theory of modular forms to prove two `isolated' congruences modulo 3 and 5, respectively. 
	\begin{theorem}[{{\cite[Theorem~1.4]{thfa}}}]\label{t1.2}
		For $\alpha\ge0$ and all $n\ge0$,
		\begin{align*}
			a_5\displaystyle\left(3^{2\alpha+3}n+\frac{153\cdot3^{2\alpha}-1)}{8}\right)\equiv0\pmod3.
		\end{align*}
	\end{theorem}
	
	\begin{theorem}[{{\cite[Theorem~1.5]{thfa}}}]\label{t1.3}For all  $n\ge0$,
		\begin{align*}
			a_5\displaystyle\left(5n+3\right)\equiv0\pmod5.
		\end{align*}
	\end{theorem}
	With the initial goal to provide elementary proofs of Theorem \ref{t1.2} and \ref{t1.3}, Sellers \cite{sellers} extended Theorem \ref{t1.2} as follows.
	\begin{theorem}[{{\cite[Theorem~4.1]{sellers}}}]\label{t1.4}
		For $0\le t\le 8$ and all $n\ge0$,
		\begin{align*}
			a_{3t+2}(27n+(18+t))\equiv 0 \pmod 3.
		\end{align*}
	\end{theorem}
	\noindent The above theorem implies that, for all $n\geq 0$, 
	\begin{align}\label{e1.6}
		a_5(27n+19)\equiv 0 \pmod 3.
	\end{align}
	Sellers \cite{sellers} established the following internal congruence modulo 3 for $a_5(n)$. For all $n\ge0$, we have
	\begin{align}\label{e1.5}
		a_{5}(27n+10)\equiv a_{5}(3n+1)\pmod 3.
	\end{align}
	
	Very recently, Amdeberhan, Sellers, and Singh \cite{ajsi} introduced partitions with multicolored even parts and monochromatic odd parts. We denote the number of such generalized colored partitions by $a_r^*(n)$. The generating function for $a_r^*(n)$ is given by
	\begin{align*}
		\sum_{n=0}^{\infty}a_r^*(n)q^n=\dfrac{1}{f_1f_2^{r-1}}\; .
	\end{align*}
	They established several congruences modulo primes satisfied by $a_r^*(n)$ as follows:
	\begin{theorem}[{{\cite[Theorem~1.2]{ajsi}}}]
		Let $p$ be an odd prime. Then, for all $n\ge 0$,
		\begin{align*}
			a_{p-1}^*(pn+r)\equiv 0 \pmod p
		\end{align*}
		where $r$ is an integer, $1\le r\le p-1$, such that $8r+1$ is a quadratic nonresidue modulo $p$.
	\end{theorem}
	\noindent Subsequently, Das, Maity, and Saikia \cite{hdas}, Dockery \cite{doc}, and Guadalupe \cite{rusg, rusg2, rus3} obtained several infinite families of congruences for $a_r^*(n)$.\\
	\indent Our aim in this paper is to significantly extend the definition of color partitions. In order to describe our approach we first consider a natural combinatorial generalization, namely, partitions with $r$-colored even parts and $s$-colored odd parts. We denote the number of such generalized colored partitions by $a_{r,s}(n)$ and define $a_{r,s}(0):=1$. The generating function for $a_{r,s}(n)$ is given by
	\begin{align}\label{e1.4}
		\sum_{n=0}^{\infty}a_{r,s}(n)q^n=\dfrac{f_2^{s-r}}{f_1^s}\; .
	\end{align}
	This provides a natural generalization to several of the generating functions above for different color partitions.
	
	This article contributes to the study of the family of functions \eqref{e1.4} in detail by establishing the following results. 
	\begin{theorem}\label{t1.6}
		For all $n\ge0$, and all $k\ge j\geq 0$,
		\begin{align*}
			a_{3(k-j)+3,\; 3k+6}(3n+1)\equiv a_{3(k-j)+3,\; 3k+6}(3n+2) \equiv 0 \pmod3. 
		\end{align*}
	\end{theorem}
	

	\begin{theorem}\label{t1.8}
		For all $n\ge0$, and all $k\ge j\geq 0$,
		\begin{align*}
			a_{27(k-j)+2,\; 27k+4}(27n+9)\equiv a_{27(k-j)+2,\; 27k+4}(27n+18)\equiv 0 \pmod 3.	
		\end{align*}
	\end{theorem}

	\begin{theorem}\label{t1.9}
		For all $n\ge 0$,
		\begin{align*}
			a_{2,4}(27n)\equiv	a_{2,4}(3n) \pmod 3.
		\end{align*} 
	\end{theorem}

	\begin{theorem}\label{t1.10}
		For all $n\ge0$, and all $k\ge j\geq 0$,
		\begin{align*}
			a_{9(k-j)+4,\;9k+2}(9n+7)\equiv 0 \pmod 3.
		\end{align*}
	\end{theorem}	

	\begin{theorem}\label{t1.11}
		For all $n\ge0$, and all $k\ge j\geq 0$,
		\begin{align*}
			a_{3(k-j)+5,\;3k+1}(3n+2)\equiv 0 \pmod 3.
		\end{align*}
	\end{theorem}	
	\begin{theorem}\label{t1.12}
		For all $n\ge0$, and all $k\ge j\geq 0$,
		\begin{align*}
			a_{3(k-j)+5,\;3k+1}(9n+6)\equiv 0 \pmod 3.
		\end{align*}
	\end{theorem}	

	\begin{theorem}\label{t1.13}
		Let $p\equiv 5, 11\pmod {12}$ be a prime and $r$,  $1\le r\le p-1$, such that $p\mid 3r+1$. Then for all $n\ge 0, k\ge j\ge 0$,
		\begin{align}\label{en1.13}
			a_{p(k-j)+(p-4),\;pk+p}(pn+r)\equiv 0 \pmod p.
		\end{align}
	\end{theorem}	

	
	\begin{theorem}\label{t1.15} For all $n\ge0$, $k\ge j\ge0$, primes $p\ge 5$, and all $r$, $1\le r\le p-1$, such that $4r+1$ is a quadratic non-residue modulo $p$,
		\begin{align}
			a_{p(k-j)+(p-3),\;pk+p}(pn+r)&\equiv 0\pmod p. \label{e1.153}
		\end{align}
	\end{theorem}
	
	\begin{theorem}\label{t1.15b} Let $p$ prime, $p\ge 5$, and let $r$, $1\le r\le p-1$, such that $4r+1\equiv 0\pmod p$. Then for all $n\ge0$, $k\ge j\ge0$,
		\begin{align}
			a_{p(k-j)+(p-3),\;pk+p}(pn+r)&\equiv 0\pmod p. \label{e1.152}
		\end{align}
	\end{theorem}
	\begin{theorem}\label{t1.16} Let $p$ be prime, $p\ge5$, and let $r$, $1\le r\le p-1$, be a quadratic non-residue modulo $p$. Then for all $n\ge0$, $k \ge j\ge0$,
		\begin{align}
			a_{p(k-j)+(p-1),\;pk+(p-2)}(pn+r)&\equiv 0\pmod p. \label{e1.161}
		\end{align}
	\end{theorem}
	
	\begin{theorem}\label{t1.18}Let $p$ be prime, $p\ge5$, and let $r$, $1\le r\le p-1$, be such that $12r+1$ is a quadratic non-residue modulo $p$. Then for all $n\ge0$, $k\ge j\ge0$,
	\begin{align}
	a_{p(k-j)+(p-1),\;pk+p}(pn+r)&\equiv 0\pmod p. \label{e1.18}
	\end{align}
	\end{theorem}
	
	\begin{theorem}\label{t1.19}Let $p$ be prime, $p\ge3$, and let $r$, $1\le r\le p-1$, such that $8r+1$ is a quadratic non-residue modulo $p$. Then for all $n\ge0$, $k\ge j\ge0$,
		\begin{align}
			a_{p(k-j)+(p-1),\;pk+(p+1)}(pn+r)&\equiv 0\pmod p. \label{e1.19}
		\end{align}
	\end{theorem}
		\begin{theorem}\label{t1.20} For all $n\ge0$, $k\ge j\ge0$, primes $p\ge 5$, and all $r$, $1\le r\le p-1$, such that $3r+1$ is a quadratic non-residue modulo $p$ or $3r+1\equiv 0\pmod p$,
		\begin{align}
			a_{p(k-j)+(p-3),\;pk+(p+2)}(pn+r)&\equiv 0\pmod p. \label{e1.20}
		\end{align}
	\end{theorem}
	This paper is organized as follows. In Section \ref{s2}, we collect some preliminary results required to prove our main results. In Section \ref{s3}, we give alternate proof of congruences \eqref{e1.6} and \eqref{e1.5}. In Sections \ref{s4} and \ref{s5}, we prove the various results mentioned above. 
	\section{Preliminaries}\label{s2}
	In this section, we present some basic definitions and preliminary results on Ramanujan's theta functions. Ramanujan's general theta function is defined as
	\begin{align}
		f(a,b)=\sum_{n=-\infty}^{\infty}a^{n(n+1)/2}b^{n(n-1)/2},\;\; |ab|<1 \label{e2.1}. 
	\end{align}
	The well-known Jacobi triple product identity {{\cite[Entry 19, p.35]{rnb}}} is given by
	\begin{align*}
		f(a,b)=(-a;ab)_\infty(-b;ab)_\infty(ab;ab)_\infty.
	\end{align*}
	We also require particular cases of \eqref{e2.1}:
	\begin{align}
		\phi(q):&=f(q,q)=\sum_{n=-\infty}^{\infty}q^{n^2}=1+\sum_{n=1}^{\infty}q^{n^2} \label{et1}\\
		\psi(q):&=f(q,q^3)=\sum_{n=0}^{\infty}q^{n(n+1)/2}=\dfrac{f_2^2}{f_1}\label{et2}\\
		\phi(-q)&=\dfrac{f_1^2}{f_2}\label{et3}.
	\end{align}
	\begin{lemma}[{{\cite[Equation~1.6.1]{poq}}}]\label{lepnt}
		We have 
		\begin{align*}
			f_1=\sum_{k=-\infty}^{\infty}(-1)^kq^{(3k^2-k)/2}.
		\end{align*}
	\end{lemma}
	\begin{lemma}[{{\cite[Equation~1.7.1]{poq}}}]\label{l2.7}
		We have 
		\begin{align*}
			f_1^3=\sum_{k=0}^{\infty}(-1)^k(2k+1)q^{k(k+1)/2}.
		\end{align*}
	\end{lemma}
	\begin{lemma}[{{\cite[Equation~10.7.7]{poq}}}]\label{lepq} We have
	\begin{align*}
	\dfrac{f_2^5}{f_1^2}=\sum_{k=-\infty}^{\infty}(-1)^k(3k+1)q^{3k^2+2k}.
	\end{align*}
	\end{lemma}
	We now collect the dissections which will be necessary in our work below.
	\begin{lemma}\label{l2.3a}
		The following 3-dissections hold:
		\begin{align}
			\dfrac{f_2^3}{f_1^3}&= \dfrac{f_6}{f_3}+3q\;\dfrac{f_6^4f_9^5}{f_3^8f_{18}}+6q^2\;\dfrac{f_6^3f_9^2f_{18}^2}{f_3^7}+12q^3\;\dfrac{f_6^2f_{18}^5}{f_3^6f_9}\;\label{e4.3} \\
			\dfrac{1}{f_1f_2}&=\dfrac{f_9^9}{f_3^6f_6^2f_{18}^3}+q\;\dfrac{f_9^6}{f_3^5f_6^3}+3q^2\;\dfrac{f_9^3f_{18}^3}{f_3^4f_6^4}-2q^3\;\dfrac{f_{18}^6}{f_3^3f_6^5}+4q^4\;\dfrac{f_{18}}{f_3^2f_6^6f_9^3}\;\label{e4.4}\\
			f_1f_2&=\dfrac{f_6f_9^4}{f_3f_{18}^2}-q\;f_9f_{18}-2q^2\;\dfrac{f_3f_{18}^4}{f_6f_9^2}\label{e4.5}\\
			\dfrac{f_1^2}{f_2}&= \dfrac{f_9^2}{f_{18}}-2q\;\dfrac{f_3f_{18}^2}{f_6f_{9}}\label{e2.5} \\
			\psi(q)&=P(q^3)+q\psi(q^9)\label{el2.9}\\
			\dfrac{1}{\psi(q)}&=\dfrac{\psi(q^9)}{\psi(q^3)^4}\displaystyle\left(P(q^3)^2-qP(q^3)\psi(q^9)+q^2\psi(q^9)^2\right), \label{el2.10}
		\end{align}
		where $P(q)=\dfrac{f_2f_3^2}{f_1f_6}$.
		\begin{proof}
			The identity \eqref{e4.3} was proved by Toh \cite{dis1}. The identities \eqref{e4.4} and \eqref{e4.5} were obtained by Chern and Hao \cite{dis2}, and Hirschhorn and Sellers \cite{dis3}, respectively. The identities \eqref{e2.5} and \eqref{el2.10} were proved by Hirschhorn and Sellers \cite{hs2010}, while for identity \eqref{el2.9} see {{\cite[Equation 14.3.5]{poq}}}.
		\end{proof}
	\end{lemma}
	Employing the Binomial Theorem, we can easily establish the following important congruence, which will be frequently used without explicitly mentioning.
	
	\begin{lemma}\label{bt}
		For any prime $p$ and positive integers $k$ and $m$,
		\begin{align*}
			f_m^{p^k} \equiv f_{mp}^{p^{k-1}} \pmod{p^k}.
		\end{align*}
		\begin{proof}
			See {{\cite[Lemma 3]{bino}}} for a proof.
		\end{proof}
	\end{lemma}
	\noindent We close this section with key machinery in the following theorem.
	\begin{lemma}\label{l2.5}
		For all  primes $p$, and positive integers $\alpha_i, \beta_i, \gamma_i, \delta_i, a_i, b_i, \text{and} \;\lambda$, define \\
		\begin{align*}
			\displaystyle{\sum_{n=0}^{\infty}A(n)q^n:=\dfrac{\prod_{i=1}^{j}f_{\alpha_i}^{\gamma_i}}{\prod_{i=1}^{k}f_{\beta_i}^{\delta_i}}}
			\quad	\text{and} \quad \displaystyle{\sum_{n=0}^{\infty}B(n)q^n:=\dfrac{\prod_{i=1}^{l}f_{\alpha_i}^{a_ip^\lambda+\gamma_i}}{\prod_{i=1}^{m}f_{\beta_i}^{b_ip^\lambda+\delta_i}}}. \\
		\end{align*}		
For $1\le C\le p^\lambda -1$, if $A(p^\lambda n+C)\equiv 0\pmod p$ for all $n\geq 0$, then $B(p^\lambda n+C)\equiv 0\pmod p$ for all $n\geq 0$.
		\begin{proof}
			We have 
			\begin{align*}
				\sum_{n=0}^{\infty}B(n)q^n=\dfrac{\prod_{i=1}^{l}f_{\alpha_i}^{a_ip^\lambda}}{\prod_{i=1}^{m}f_{\beta_i}^{b_ip^\lambda}}\cdot\dfrac{\prod_{i=1}^{j}f_{\alpha_i}^{\gamma_i}}{\prod_{i=1}^{k}f_{\beta_i}^{\delta_i}}.
			\end{align*}
			We use Lemma \ref{bt} to obtain
			\begin{align*}
				\sum_{n=0}^{\infty}B(n)q^n\equiv\dfrac{\prod_{i=1}^{l}f_{p^\lambda\alpha_i}^{a_i}}{\prod_{i=1}^{m}f_{p^\lambda\beta_i}^{b_i}}\sum_{n=0}^{\infty}A(n)q^n\pmod p.
			\end{align*}
			Extracting the terms in which the exponents of $q$ are of the form $p^\lambda n+C$, we get
			\begin{align*}
				\sum_{n=0}^{\infty}B(p^\lambda n+C)q^{p^\lambda n +C}\equiv\dfrac{\prod_{i=1}^{l}f_{p^\lambda \alpha_i}^{a_i}}{\prod_{i=1}^{m}f_{p^\lambda\beta_i}^{b_i}}\sum_{n=0}^{\infty}A(p^\lambda n+C)q^{p^\lambda n +C}\pmod p.
			\end{align*}
			Dividing by $q^C$ and replacing $q^{p^\lambda}$ by $q$ yields
			\begin{align}\label{e2.9}
				\sum_{n=0}^{\infty}B(p^\lambda n+C)q^{n}\equiv\dfrac{\prod_{i=1}^{l}f_{p^\lambda\alpha_i}^{a_i}}{\prod_{i=1}^{m}f_{p^\lambda\beta_i}^{b_i}}\sum_{n=0}^{\infty}A(p^\lambda n+C)q^{n}\pmod p.
			\end{align}
A straightforward interpretation of the factor $\dfrac{\prod_{i=1}^{l}f_{p^\lambda\alpha_i}^{a_i}}{\prod_{i=1}^{m}f_{p^\lambda\beta_i}^{b_i}}$ implies it is a function of $q^{p^\lambda}$. Invoking $A(p^\lambda n+C)\equiv 0\pmod p$ in \eqref{e2.9}, the proof of Lemma \ref{l2.5} is completed.
		\end{proof}
	\end{lemma}

\section{Alternate proofs for (\ref{e1.6}) and (\ref{e1.5})}
\label{s3}
	In this section, with the tools mentioned above, we provide an alternate proof for congruences \eqref{e1.6} and \eqref{e1.5}. All congruences in this section hold modulo 3.
	\begin{proof}[\textbf{Proof of (\ref{e1.6})}]
		Thanks to \eqref{e1.4}, we have
		\begin{align*}
			\sum_{n=0}^{\infty}a_{1,5}(n)q^n=\dfrac{f_2^4}{f_1^5}=\dfrac{f_2^6}{f_1^6}\dfrac{f_1}{f_2^2}\equiv \dfrac{f_6^2}{f_3^2}\dfrac{1}{\psi(q)}.
		\end{align*}
		From \eqref{el2.10}, we have
		\begin{align*}
			\sum_{n=0}^{\infty}a_{1,5}(3n+1)q^{3n+1}\equiv \dfrac{f_6^2}{f_3^2}\dfrac{\psi(q^9)}{\psi(q^3)^4}\left(-qP(q^3)\psi(q^9)\right).
		\end{align*}
		This implies,
		\begin{align}
			\sum_{n=0}^{\infty}a_{1,5}(3n+1)q^{n}&\equiv -\dfrac{f_2^2}{f_1^2}\dfrac{\psi(q^3)^2}{\psi(q)^4}P(q)\nonumber\\
			&\equiv -\dfrac{\psi(q^3)}{\psi(q)}\dfrac{f_2^2}{f_1^2}\dfrac{f_2f_3^2}{f_1f_6}\nonumber\\
			&=-\psi(q^3)\dfrac{f_2^3f_3^2}{f_1^3f_6}\dfrac{1}{\psi(q)}\nonumber\\
			&\equiv -\psi(q^3)f_3\dfrac{1}{\psi(q)}\label{e3.1}.
		\end{align}
		Again, applying \eqref{el2.10} on \eqref{e3.1}, and then extracting those terms in which the power of $q$ is congruent to 0 modulo 3, and replacing $q^3$ by $q$, we obtain
		\begin{align*}
			\sum_{n=0}^{\infty}a_{1,5}(9n+1)q^{n}&\equiv-f_1\dfrac{\psi(q^3)}{\psi(q)^3}P(q)^2\\
			&\equiv-f_1 P(q)^2\\
			&=-f_1\displaystyle\left(\dfrac{f_2f_3^2}{f_1f_6}\right)^2\\
			&=-\dfrac{f_3^4}{f_6^2}\psi(q).
		\end{align*}
		From \eqref{el2.9}, we have			
		\begin{equation}
                \label{a15_91_diss}
			\sum_{n=0}^{\infty}a_{1,5}(9n+1)q^{n}\equiv-\dfrac{f_3^4}{f_6^2}\;[P(q^3)+q\psi(q^9)].
		\end{equation}
		To complete the proof, it suffices to note that	there are no terms of the form $q^{3n+2}$ in this last congruence. 
	\end{proof}
	\begin{corollary}
		For all $n\ge0$, and all $k\ge j\geq 0$,
		\begin{align*}
			a_{27(k-j)+1,\; 27k+5}(27n+19)\equiv 0\pmod 3.
		\end{align*}
		\begin{proof}
			For all $n\ge0$, and all $k\ge j\geq 0$, we define 
			\begin{align*}
				A(n):&=a_{1,5}(n)	\\
				B(n):&=a_{27(k-j)+1,\;27k+5}(n).
			\end{align*} Employing Lemma \ref{l2.5}, with $(p,\lambda)=(3,3)$ and $C=19$, we complete the proof.
		\end{proof}
	\end{corollary}
	
	\begin{proof}[\textbf{Proof of (\ref{e1.5})}]

From \eqref{a15_91_diss}, we have 
$$\sum_{n=0}^{\infty}a_{1,5}(9n+1)q^n\equiv -\dfrac{f_3^4}{f_6^2}[P(q^3)+q\psi(q^9)].$$
		By further extracting the terms in which the exponents of $q$ are of the form $3n+1$, we get
		\begin{align*}
			\sum_{n=0}^{\infty}a_{1,5}(27n+10)q^{3n+1}&\equiv -\dfrac{f_3^4}{f_6^2}q\psi(q^9),
           \end{align*}
or 
\begin{align*}
\sum_{n=0}^{\infty}a_{1,5}(27n+10)q^{n}
&\equiv 
-\dfrac{f_1^4}{f_2^2}\psi(q^3) \\
&= 
-f_1^3\dfrac{f_1}{f_2^2}\psi(q^3) \\
&\equiv 
-\psi(q^3)f_3\frac{1}{\psi(q)} \\
&\equiv 
\sum_{n=0}^{\infty}a_{1,5}(3n+1)q^{n}.
\end{align*}
Thanks to (\ref{e3.1}).  
		This completes the proof.
	\end{proof}
	
	\begin{corollary}\label{c3.2}
		For all $n\ge0$, and all $k\ge j\geq 0$,
		\begin{align*}
			a_{27(k-j)+1,\; 27k+5}(27n+10)\equiv	a_{3(k-j)+1,\;3k+5}(3n+1)\pmod 3.
		\end{align*}
	\end{corollary}
	\begin{proof}
		We have
		\begin{align*}
			\sum_{n=0}^{\infty}a_{27(k-j)+1,\;27k+5}(n)q^n&=\dfrac{f_2^{27j+4}}{f_1^{27k+5}}\\
			&\equiv \dfrac{f_{54}^j}{f_{27}^k} \sum_{n=0}^{\infty}a_{1,5}(n)q^n.
		\end{align*}
		This implies,
		\begin{align*}
			\sum_{n=0}^{\infty}a_{27(k-j)+1,\;27k+5}(27n+10)q^n\equiv \dfrac{f_2^j}{f_1^k}\sum_{n=0}^{\infty}a_{1,5}(27n+10)q^n .
		\end{align*}
		From \eqref{e1.5}, we have
		\begin{align*}
			\sum_{n=0}^{\infty}a_{27(k-j)+1,\;27k+5}(27n+10)q^n\equiv \dfrac{f_2^j}{f_1^k}\sum_{n=0}^{\infty}a_{1,5}(3n+1)q^n .
		\end{align*}
		Now,
		\begin{align*}
			\sum_{n=0}^{\infty}a_{3(k-j)+1,\;3k+5}(n)q^n&=\dfrac{f_2^{3j+4}}{f_1^{3k+5}}\\
			&\equiv \dfrac{f_{6}^j}{f_{3}^k} \sum_{n=0}^{\infty}a_{1,5}(n)q^n.
		\end{align*}
		This implies,
		\begin{align*}
			\sum_{n=0}^{\infty}a_{3(k-j)+1,\;3k+5}(3n+1)q^n\equiv \dfrac{f_2^j}{f_1^k}\sum_{n=0}^{\infty}a_{1,5}(3n+1)q^n.
		\end{align*}
		Therefore, for all $n\geq 0$, 
		\begin{align*}
			a_{27(k-j)+1,\; 27k+5}(27n+10)\equiv	a_{3(k-j)+1,\;3k+5}(3n+1)\pmod3.
		\end{align*}
		This completes the proof of Corollary \ref{c3.2}.
	\end{proof}
	\begin{remark}
		Applying induction on \eqref{e1.6} and \eqref{e1.5}, we obtain Theorem \ref{t1.2}.
	\end{remark}

\section{Proofs of Theorems \ref{t1.6} -- \ref{t1.12}}\label{s4}
	In this section, all congruences hold modulo 3.
	\begin{proof}[\textbf{Proof of Theorem \ref{t1.6}}]
		From \eqref{e1.4}, we have
		\begin{align*}
			\sum_{n=0}^{\infty}a_{3,6}(n)q^n&=\dfrac{f_2^3}{f_1^6}\\
			&\equiv\dfrac{f_6}{f_3^2}.
		\end{align*} 
		It is easy to see the last expression is a function of $q^3$.  Therefore, for all $n\geq 0$, we obtain
		\begin{align*}
			a_{3,6}(3n+1)&\equiv 0 \pmod3 \\
			a_{3,6}(3n+2)&\equiv 0 \pmod3.
		\end{align*}
		Now, for all $n\ge0$, and all $k\ge j\geq 0$, we define 
		\begin{align*}
			A(n):&=a_{3,6}(n)	\\
			B(n):&=a_{3(k-j)+3,\;3k+6}(n).
		\end{align*}
		On employing Lemma \ref{l2.5}, with $(p,\lambda)=(3,1)$ and $C=1$ and $2$, we establish Theorem \ref{t1.6}.
	\end{proof}

	\begin{proof}[\textbf{Proof of Theorem \ref{t1.8}}]
		Thanks to \eqref{e1.4} and \eqref{el2.9}, we have
		\begin{align*}
			\sum_{n=0}^{\infty} a_{2,4}(n)q^n&= \dfrac{f_2^2}{f_1^4}\\
			&\equiv\dfrac{1}{f_3}\psi(q)\\		
			&\equiv\dfrac{1}{f_3}\displaystyle\left(\dfrac{f_6f_9^2}{f_3f_{18}}+q\dfrac{f_{18}^2}{f_9} \right).
		\end{align*} 
		Extracting the terms in which the exponents of $q$ are of the form $3n$ and replacing $q^3$ by $q$, we get
		\begin{align}
			\sum_{n=0}^{\infty} a_{2,4}(3n)q^{n}&\equiv \dfrac{f_2f_3^2}{f_1^2f_6} \label{e4.1}\\ 
			&=\dfrac{f_3^2}{f_6}\sum_{n=0}^{\infty}\bar{p}(n)q^n, \nonumber
		\end{align}
		where $\bar{p}(n)$ enumerates the number of overpartitions of $n$ \cite{love}.
		Also from {{\cite[proof of Theorem 2.1]{hs2005}}},
		\begin{align*}
			\sum_{n=0}^{\infty} a_{2,4}(9n)q^{3n}&\equiv \dfrac{f_3^2}{f_6}\sum_{n=0}^{\infty}\bar{p}(3n)q^{3n}\\
			&\equiv \dfrac{f_9^4}{f_{18}^2}.
		\end{align*}
		Thus, \begin{align}\label{e4.6a}
			\sum_{n=0}^{\infty} a_{2,4}(9n)q^{n}\equiv\dfrac{f_3^4}{f_6^2}.
		\end{align}
		Observe that $\dfrac{f_3^4}{f_6^2}$ is a function of $q^3$. Therefore, for all $n\geq 0$, we must have 
		\begin{align}
			a_{2,4}(9(3n+1))&=a_{2,4}(27n+9)\equiv 0 \label{e41.3}\\
			a_{2,4}(9(3n+2))&=a_{2,4}(27n+18)\equiv0 \label{e41.4}.
		\end{align}
		Now, for all $n\ge0$, and all $k\ge j\geq 0$, we define 
		\begin{align*}
			A(n):&=a_{2,4}(n)	\\
			B(n):&=a_{27(k-j)+2,\;27k+4}(n).
		\end{align*} Employing Lemma \ref{l2.5}, with $(p,\lambda)=(3,3)$ and $C=9$ and $18$, we  establish Theorem \ref{t1.8}.
	\end{proof}

	\begin{proof}[\textbf{Proof of Theorem \ref{t1.9}}]
		Thanks to \eqref{e4.1}, we have
		\begin{align*}
			\sum_{n=0}^{\infty}a_{2,4}(3n)q^n\equiv\dfrac{f_1^4}{f_2^2}.
		\end{align*}
		Now, in \eqref{e4.6a}, we extract the terms in which the exponents of $q$ are of the form $3n$, to obtain
		\begin{align*}
			\sum_{n=0}^{\infty}a_{2,4}(27n)q^n&\equiv\dfrac{f_1^4}{f_2^2}.
		\end{align*}
		Theorem \ref{t1.9} follows immediately.
	\end{proof}
	
	\begin{corollary}\label{c4.1}
		For all $n\ge0$, and all $k\ge j\geq 0$,
		\begin{align*}
			a_{27(k-j)+2,\; 27k+4}(27n)\equiv	a_{3(k-j)+2,\;3k+4}(3n).
		\end{align*}
	\end{corollary}
	\begin{proof}
		We have
		\begin{align*}
			\sum_{n=0}^{\infty}a_{27(k-j)+2,\;27k+4}(n)q^n&=\dfrac{f_2^{27j+2}}{f_1^{27k+4}}\\
			&\equiv \dfrac{f_{54}^j}{f_{27}^k} \sum_{n=0}^{\infty}a_{2,4}(n)q^n ,
		\end{align*}
		which implies
		\begin{align*}
			\sum_{n=0}^{\infty}a_{27(k-j)+2,\;27k+4}(27n)q^n\equiv \dfrac{f_2^j}{f_1^k}\sum_{n=0}^{\infty}a_{2,4}(27n)q^n.
		\end{align*}
		From Theorem \ref{t1.9}, we have
		\begin{align*}
			\sum_{n=0}^{\infty}a_{27(k-j)+2,\;27k+4}(27n)q^n\equiv \dfrac{f_2^j}{f_1^k}\sum_{n=0}^{\infty}a_{2,4}(3n)q^n.
		\end{align*}
		Again, we have
		\begin{align*}
			\sum_{n=0}^{\infty}a_{3(k-j)+2,\;3k+4}(n)q^n&=\dfrac{f_2^{3j+2}}{f_1^{3k+4}}\\
			&\equiv \dfrac{f_{6}^j}{f_{3}^k} \sum_{n=0}^{\infty}a_{2,4}(n)q^n.
		\end{align*}
		Extracting the terms in which the exponents of $q$ are of the form $3n$, we obtain
		\begin{align*}
			\sum_{n=0}^{\infty}a_{3(k-j)+2,\;3k+4}(3n)q^n\equiv \dfrac{f_2^j}{f_1^k}\sum_{n=0}^{\infty}a_{2,4}(3n)q^n .
		\end{align*}
		Thanks to Theorem \ref{t1.9}, this completes the proof of Corollary \ref{c4.1}.
	\end{proof}

	\begin{corollary}\label{c4.2}
		For all $n\ge0$, and all $k\ge j\geq 0$,
		\begin{align*}
			a_{2,4}(3^{2\alpha+3}n+3^{2(\alpha+1)})\equiv 0 \pmod 3\\
			a_{2,4}(3^{2\alpha+3}n+2\cdot3^{2(\alpha+1)})\equiv 0 \pmod3.
		\end{align*}
	\end{corollary}
	\begin{proof}
		Employing induction on \eqref{e41.3} and \eqref{e41.4}, along with Theorem \ref{t1.9}, we complete the proof of  Corollary \ref{c4.2}.
	\end{proof}
	
	\begin{proof}[\textbf{Proof of Theorem \ref{t1.10}}]
		From \eqref{e1.4}, we have
		\begin{align*}
			\sum_{n=0}^{\infty}a_{4,2}(n)q^n&=\dfrac{1}{f_1^2f_2^2}\equiv \dfrac{f_1f_2}{f_3f_6}.
		\end{align*} 
		Using \eqref{e4.5}, we have
		\begin{align*}
			\sum_{n=0}^{\infty}a_{4,2}(n)q^n\equiv \dfrac{1}{f_3f_6}\displaystyle\left(\dfrac{f_6f_9^4}{f_3f_{18}^2}-qf_9f_{18}-2q^2\;\dfrac{f_3f_{18}^4}{f_6f_9^2}\right).
		\end{align*}
		We now extract the terms in which the exponents of $q$ are of the form $3n+1$, divide both sides by $q$, and then replace $q^3$ by $q$ to obtain
		\begin{align*}
			\sum_{n=0}^{\infty}	a_{4,2}(3n+1)q^n&\equiv -\dfrac{1}{f_1f_2}f_3f_6.
		\end{align*}
		Again, using \eqref{e4.4} in the above expression, we obtain
		\begin{align*}
			\sum_{n=0}^{\infty}	a_{4,2}(3n+1)q^n\equiv	-f_3f_6\displaystyle\left(\dfrac{f_9^9}{f_3^6f_6^2f_{18}^3}+q\;\dfrac{f_9^6}{f_3^5f_6^3}+3q^2\;\dfrac{f_9^3f_{18}^3}{f_3^4f_6^4}-2q^3\;\dfrac{f_{18}^6}{f_3^3f_6^5}+4q^4\;\dfrac{f_{18}}{f_3^2f_6^6f_9^3}\right).
		\end{align*}
		Extracting the terms in which the exponents of $q$ are of the form $3n+2$, we conclude that
		\begin{align*}
			a_{4,2}(9n+7)&\equiv 0\pmod3.
		\end{align*}
		Now, for all $n\ge0$, and all $k\ge j\geq 0$, we define
		\begin{align*}
			A(n):&=a_{4,2}(n)	\\
			B(n):&=a_{9(k-j)+4,\;9k+2}(n).
		\end{align*} Employing Lemma \ref{l2.5}, with $(p,\lambda)=(3,2)$ and $C=7$, we complete the proof of Theorem \ref{t1.10}.
	\end{proof}
	
	\begin{proof}[\textbf{Proof of Theorem \ref{t1.11}}]
		From \eqref{e1.4}, we have
		\begin{align*}
			\sum_{n=0}^{\infty}a_{5,1}(n)q^n&=\dfrac{1}{f_1f_2^4}\equiv \dfrac{1}{f_1f_2}\dfrac{1}{f_6}.
		\end{align*} 
		Employing \eqref{e4.4}, we have 
		\begin{align*}
			\sum_{n=0}^{\infty}a_{5,1}(n)q^n\equiv\dfrac{1}{f_6}\displaystyle\left(\dfrac{f_9^9}{f_3^6f_6^2f_{18}^3}+q\;\dfrac{f_9^6}{f_3^5f_6^3}+3q^2\;\dfrac{f_9^3f_{18}^3}{f_3^4f_6^4}-2q^3\;\dfrac{f_{18}^6}{f_3^3f_6^5}+4q^4\;\dfrac{f_{18}}{f_3^2f_6^6f_9^3}\right).\hspace{-0.4cm}
		\end{align*}
		Extracting the terms in which the exponents of $q$ are of the form $3n+2$, we conclude that
		\begin{align*}
			a_{5,1}(3n+2)&\equiv 0 \pmod3.
		\end{align*}
		Again, for all $n\ge0$, and all $k\ge j\geq 0$, we define
		\begin{align*}
			A(n):&=a_{5,1}(n)	\\
			B(n):&=a_{3(k-j)+5,\;3k+1}(n).
		\end{align*} Employing Lemma \ref{l2.5}, with $(p,\lambda)=(3,1)$ and $C=2$, we complete the proof of Theorem \ref{t1.11}.
	\end{proof}
	
	\begin{proof}[\textbf{Proof of Theorem \ref{t1.12}}]
		From \eqref{e1.4}, we have
		\begin{align*}
			\sum_{n=0}^{\infty}a_{5,1}(n)q^n&=\dfrac{1}{f_1f_2^4}\equiv \dfrac{1}{f_3f_6}\dfrac{f_1^2}{f_2}.
		\end{align*} 
		Employing \eqref{e2.5}, we have
		\begin{align*}
			\sum_{n=0}^{\infty}a_{5,1}(n)q^n&\equiv \dfrac{1}{f_3f_6}\displaystyle\left(\dfrac{f_9^2}{f_{18}}-2q\;\dfrac{f_3f_{18}^2}{f_6f_{9}}\right).
		\end{align*} 
		Extracting the terms in which the exponents of $q$ are of the form $3n$, we get
		\begin{align*}
			\sum_{n=0}^{\infty}	a_{5,1}(3n)q^n&\equiv \dfrac{1}{f_1f_2}\dfrac{f_3^2}{f_6}.
		\end{align*}
		We now use \eqref{e4.4} to obtain
		\begin{align*}
			\sum_{n=0}^{\infty}	a_{5,1}(3n)q^n\equiv \dfrac{f_3^2}{f_6}\displaystyle\left(\dfrac{f_9^9}{f_3^6f_6^2f_{18}^3}+q\;\dfrac{f_9^6}{f_3^5f_6^3}+3q^2\;\dfrac{f_9^3f_{18}^3}{f_3^4f_6^4}-2q^3\;\dfrac{f_{18}^6}{f_3^3f_6^5}+4q^4\;\dfrac{f_{18}}{f_3^2f_6^6f_9^3}\right).
		\end{align*}
		Extracting the terms in which the exponents of $q$ are of the form $3n+2$, we conclude that
		\begin{align*}
			a_{5,1}(9n+6)&\equiv 0 \pmod 3.
		\end{align*}
		Again, for all $n\ge0$, and all $k\ge j\geq 0$, we define
		\begin{align*}
			A(n):&=a_{5,1}(n)	\\
			B(n):&=a_{9(k-j)+5,\;9k+1}(n).
		\end{align*} Employing Lemma \ref{l2.5}, with $(p,\lambda)=(3,2)$ and $C=6$, we complete the proof of Theorem \ref{t1.12}.
	\end{proof}
	
\section{Proofs of Theorems \ref{t1.13} -- \ref{t1.20}}\label{s5}
In this section, all congruences hold modulo the prime $p$ as referenced in each of the theorems.
		\begin{proof}[\textbf{Proof of Theorem \ref{t1.13}}]
		From \eqref{e1.4} and Lemmas \ref{lepnt} and \ref{l2.7}, we have 
		\begin{align*}
			\sum_{n=0}^{\infty}a_{p-4,\;p}(n)q^n&=\dfrac{f_2^4}{f_1^p}\\
			&\equiv \dfrac{1}{f_p}f_2 f_2^3\\
			&=\dfrac{1}{f_p}\sum_{j=-\infty}^{\infty}\sum_{k=0}^{\infty}(-1)^{j+k}(2k+1)q^{(3j^2-j)+k(k+1)}.
		\end{align*} 
		We consider $pn+r\equiv (3j^2-j)+k(k+1)\pmod p$ which is equivalent to $12r+4\equiv (6j-1)^2+3(2k+1)^2\pmod p$. It is clear that $-3$ is a quadratic non-residue modulo $p$, since $p\equiv 5$ or $11$. So we get $6j-1\equiv 2k+1\equiv 0\pmod p$, as $p\mid 3r+1$. Given the factor $(2k+1)$ inside the sum above, we prove
		\begin{align*}
			a_{p-4, p}(pn+r)&\equiv 0\pmod p,
		\end{align*}
		which is $j=k=0$ case of \eqref{en1.13}.\\
		Again, for all $n\ge0$, and all $k\ge j\geq 0$, we define
		\begin{align*}
			A(n):&=	a_{p-4, p}(n)	\\
			B(n):&=a_{p(k-j)+(p-4),\;pk+p}(n).
		\end{align*} 
		Using Lemma \ref{l2.5}
		, we complete the proof of Theorem \ref{t1.13}.
	\end{proof}

	\begin{proof}[\textbf{Proof of Theorem \ref{t1.15}}]
		From \eqref{e1.4} and Lemma \ref{l2.7}, we have
		\begin{align*}
			\sum_{n=0}^{\infty}a_{p-3,p}(n)q^n&=\dfrac{f_2^3}{f_1^p}\\
			&\equiv\dfrac{f_2^3}{f_p}\\
			&=\dfrac{1}{f_p}\sum_{n=0}^{\infty}(-1)^n(2n+1)q^{n(n+1)}.
		\end{align*}
		To prove \eqref{e1.153}, for $j=k=0$, we note that $pn+r\equiv N(N+1)\pmod p$ is true if and only if $4r+1\equiv (2N+1)^2\pmod p$. Thus, we confirm the congruence
		\begin{align*}
			a_{p-3,p}(pn+r)\equiv 0\pmod p,
		\end{align*}
		since we have assumed that $4r+1$ is a quadratic non-residue modulo p.\\
		Again, for all $n\ge0$, and all $k\ge j\geq 0$, we define
		\begin{align*}
			A(n):&=a_{p-3,p}(n)	\\
			B(n):&=	a_{p(k-j)+(p-3),\;pk+p}(pn+r).
		\end{align*} 
		Employing Lemma \ref{l2.5}
		, we complete the proof of Theorem \ref{t1.15}.
	\end{proof}
	
	\begin{proof}[\textbf{Proof of Theorem \ref{t1.15b}}]
	We have already seen that 
		\begin{align*}
		\sum_{n=0}^{\infty}a_{p-3,p}(n)q^n\equiv\dfrac{1}{f_p}\sum_{n=0}^{\infty}(-1)^n(2n+1)q^{n(n+1)}.
	\end{align*}
	Clearly, $pn+r\equiv N(N+1)\pmod p$ holds if and only if $4r+1\equiv(2N+1)^2\pmod p$, which implies $2N+1\equiv 0\pmod p$. Given the factor $(2n+1)$ inside the congruence above, the divisibility \eqref{e1.152} for $j=k=0$ holds. Therefore, we obtain
	\begin{align*}
	a_{p-3,p}(pn+r)\equiv 0\pmod p.
	\end{align*}
	Again, for all $n\ge0$, and all $k\ge j\geq 0$, we define
	\begin{align*}
		A(n):&=a_{p-3,p}(n)	\\
		B(n):&=	a_{p(k-j)+(p-3),\;pk+p}(pn+r).
	\end{align*} 
	Employing Lemma \ref{l2.5}
	 , we complete the proof of Theorem \ref{t1.15b}.
	
	\end{proof}
	
	\begin{proof}[\textbf{Proof of Theorem \ref{t1.16}}]
		From \eqref{e1.4}, \eqref{et1}, and \eqref{et3}, it follows that
		\begin{align*}
			\sum_{n=0}^{\infty}a_{p-1,p-2}(n)q^n&=\dfrac{1}{f_1^{p-2}f_2}\\
			&\equiv\dfrac{1}{f_p}\dfrac{f_1^2}{f_2}\\
			&=\dfrac{1}{f_p}\phi(-q)\\
			&=\dfrac{1}{f_p}\sum_{k=-\infty}^{\infty}(-q)^{k^2}.
		\end{align*}
		Note that $pn+r\equiv k^2 \pmod p$, which is equivalent to $r\equiv k^2\pmod p$, can never have a solution since $r$ is a quadratic non-residue modulo $p$. Hence, for all $n\geq 0$, 
		\begin{align*}
			a_{p-1,p-2}(pn+r)&\equiv 0\pmod p,
		\end{align*} 
		which is \eqref{e1.161} for $j=k=0$.\\
		Again, for all $n\ge0$, and all $k\ge j\geq 0$, we define
		\begin{align*}
			A(n):&=a_{p-1,p-2}(n)	\\
			B(n):&=	a_{p(k-j)+(p-1),\;pk+(p-2)}(n).
		\end{align*}
		Employing Lemma \ref{l2.5}
		, we complete the proof of Theorem \ref{t1.16}.
	\end{proof}	
	
	\begin{proof}[\textbf{Proof of Theorem \ref{t1.18}}]
		Thanks to \eqref{e1.4} and Lemma \ref{lepnt}, we have
		\begin{align*}
			\sum_{n=0}^{\infty}a_{p-1,p}(n)q^n&=\dfrac{f_2}{f_1^p}\\
			&\equiv \dfrac{1}{f_{p}}f_2\\
		&=\dfrac{1}{f_p}\sum_{k=-\infty}^{\infty}(-1)^kq^{3k^2-k}.
		\end{align*}
		In order to prove \eqref{e1.18}, for $j=k=0$, we consider $pn+r\equiv 3k^2-k\pmod p$ which is equivalent to $12r+1\equiv (6k-1)^2\pmod p$. This congruence fails to hold since $12r+1$ is assumed to be a quadratic non-residue modulo $p$. Thus, we deduce
		\begin{align*}
			a_{p-1,\;p}(pn+r)\equiv 0\pmod p.
		\end{align*}
		Again, for all $n\ge0$, and all $k\ge j\geq 0$, we define
		\begin{align*}
			A(n):&=a_{p-1,\;p}(n)	\\
			B(n):&=	a_{p(k-j)+(p-1),\;pk+p}(n).
		\end{align*} 
		Employing Lemma \ref{l2.5}
		, we complete the proof of Theorem \ref{t1.18}.
	\end{proof}
	
	\begin{proof}[\textbf{Proof of Theorem \ref{t1.19}}]
		From \eqref{e1.4} and \eqref{et2}, we have 
		\begin{align*}
			\sum_{n=0}^{\infty}a_{p-1,\;p+1}(n)q^n&=\dfrac{f_2^2}{f_1^{p+1}}\\
			&\equiv\dfrac{1}{f_p}\sum_{k=0}^{\infty}q^{k(k+1)/2}.
	\end{align*} 
	To establish \eqref{e1.19} for $j=k=0$, we focus on $pn+r\equiv \frac{k(k+1)}{2}\pmod p$ which is same as $8r+1\equiv (2k+1)^2\pmod p$. This congruence cannot hold, since $8r+1$ is a quadratic non-residue modulo $p$. This yields
		\begin{align*}
			a_{p-1, p+1}(pn+r)&\equiv 0\pmod p.
		\end{align*}
		Again, for all $n\ge0$, and all $k\ge j\geq 0$, we define
		\begin{align*}
			A(n):&=	a_{p-1,\; p+1}(n)	\\
			B(n):&=a_{p(k-j)+(p-1),\;pk+(p+1)}(n).
		\end{align*} 
		Employing Lemma \ref{l2.5}
		, we complete the proof of Theorem \ref{t1.19}.
	\end{proof}
	 	\begin{proof}[\textbf{Proof of Theorem \ref{t1.20}}]
	 	From \eqref{e1.4} and Lemma \ref{lepq}, we have
	 	\begin{align*}
	 		\sum_{n=0}^{\infty}a_{p-3,p+2}(n)q^n&=\dfrac{f_2^5}{f_1^{p+2}}\\
	 		&\equiv \dfrac{f_2^5}{f_pf_1^2}\\
	 		&=\dfrac{1}{f_p}\sum_{k=-\infty}^{\infty}(-1)^k(3k+1)q^{3k^2+2k}.
	 	\end{align*}
	 	We see that $3k^2+2k\equiv pn+r \pmod p$ holds if and only if $(3k+1)^2\equiv 3r+1 \pmod p$. This can never have a solution since $3r+1$ is assumed to be a quadratic non-residue modulo $p$ or $3r+1\equiv 0\pmod p$. Thus, we prove
	 	\begin{align}
	 		a_{p-3,\;p+2}(pn+r)\equiv 0 \pmod p,
	 	\end{align}
	 	which is \eqref{e1.20} for $j=k=0$.\\
	 	Again, for all $n\ge0$, and all $k\ge j\geq 0$, we define
	 	\begin{align*}
	 		A(n):&=a_{p-3,p+2}(n)	\\
	 		B(n):&=	a_{p(k-j)+(p-3),\;pk+(p+2)}(pn+r).
	 	\end{align*} 
	 	Employing Lemma \ref{l2.5}
	 	, we complete the proof of Theorem \ref{t1.20}.
	 \end{proof}

	\bigskip
	\bigskip
	
	\noindent\textsuperscript{1,3}Ramanujan School of Mathematical Sciences,\\Department of Mathematics,\\ Pondicherry University,\\ Puducherry - 605014, India. \\

	\bigskip
	\noindent\textsuperscript{2}Department of Mathematics and Statistics, \\University of Minnesota Duluth,\\Duluth, MN 55812, USA.\bigskip\\
	
	\noindent Email: \texttt{tthejithamp@pondiuni.ac.in} \\
	Email: \texttt{jsellers@d.umn.edu} (\Letter)\\
	Email: \texttt{dr.fathima.sn@pondiuni.ac.in}


\begin{thebibliography}{99}
		\bibitem{ajsi}Amdeberhan, T., Sellers, J. A., and Singh, A. (2025). Arithmetic properties for generalized cubic partitions and overpartitions modulo a prime. Aequationes Mathematicae, 99(3), 1197-1208.
		
		\bibitem{rnb}Berndt, B. C. (1991). Ramanujan’s Notebooks: Part III. 
		
		\bibitem{dis2}Chern, S., and Hao, L. J. (2019). Congruences for two restricted overpartitions.  Proc. Indian Acad. Sci. Math. Sci., 129(3), Paper 31.
		
		\bibitem{love} Corteel, S. and Lovejoy, J. (2004). Overpartitions. Transactions of the American
		Mathematical Society, 356(4), 1623-1635.
		
		\bibitem{hdas}Das, H., Maity, S., and Saikia, M. P. (2025). Arithmetic properties of generalized cubic and overcubic partitions. arXiv preprint arXiv:2503.19399.
		
		\bibitem{bino}da Silva, R. and Sellers, J. A. (2020). Infinitely many congruences for $k$-regular partitions with designated summands. Bulletin of the Brazilian Mathematical Society, New Series, 51(2), 357-370.
		
		\bibitem{doc}Dockery, D. (2025). A congruence family modulo powers of 5 for generalized cubic partitions via the localization method. arXiv preprint arXiv:2508.05833.
		
		\bibitem{rusg}Guadalupe, R. (2025). A note on congruences for generalized cubic partitions modulo primes. Integers, 25, A20.
		
		\bibitem{rusg2} Guadalupe, R. (2025). A remark on generalised cubic partitions modulo. Bulletin of the Australian Mathematical Society, 1-5.
		
		\bibitem{rus3} Guadalupe, R. (2025). Congruences modulo $7$ and $11$ for certain two restricted partition functions. arXiv preprint arXiv:2508.18286.

		\bibitem{poq}Hirschhorn, M. D. (2017). The Power of $q$. Developments in Mathematics, 49.
		
		\bibitem{hs2005}Hirschhorn, M. D. and Sellers, J. A. (2005). An infinite family of overpartition congruences modulo 12. Integers, 5(1), A20.
		
		\bibitem{hs2010}Hirschhorn, M. D. and Sellers, J. A. (2010). Arithmetic properties of partitions with odd parts distinct. The Ramanujan Journal, 22(3), 273-284.
		
		\bibitem{dis3}Hirschhorn, M. D. and  Sellers, J. A. (2014). A congruence modulo 3 for partitions into distinct non-multiples of four. J. Integer Seq., 17(9), Article 14.9.6.
				
		\bibitem{hs1}Hirschhorn, M. D. and Sellers, J. A. (2025). A family of congruences modulo 7 for partitions with monochromatic even parts and multi--colored odd parts. arXiv preprint arXiv:2507.09752.
		
		\bibitem{sellers}Sellers, J. A. (2025). Elementary proofs and generalizations of recent congruences of Thejitha and Fathima. arXiv preprint arXiv:2510.01572.
		
		\bibitem{thfa}Thejitha, M. P. and Fathima, S. N. (2025). Arithmetic properties of partitions with 1-colored even parts and $r$-colored odd parts. arXiv preprint arXiv:2509.24324.
		
		\bibitem{dis1}Toh, P. C. (2012). Ramanujan type identities and congruences for partition pairs. Discrete Mathematics, 312(6), 1244-1250.

	\end{thebibliography}
\end{document}